\documentclass{article}
\usepackage[utf8]{inputenc}
\usepackage{amsmath}
\usepackage{amssymb}
\usepackage{tikz}
\usepackage{graphicx}
\usepackage{float}
\usepackage{framed}
\usepackage[hang,flushmargin]{footmisc}
\usepackage{comment}
\usepackage{ytableau}
\usepackage[nottoc,numbib]{tocbibind}

\usepackage{amsthm}
\newtheorem{theorem}{Theorem}[section]

\newtheorem{lemma}[theorem]{Lemma}
\newtheorem{definition}[theorem]{Definition}
\newtheorem{proposition}[theorem]{Proposition}

\newtheorem{remark}[theorem]{Remark}

\usepackage{scalerel,stackengine}
\stackMath
\newcommand\reallywidehat[1]{%
\savestack{\tmpbox}{\stretchto{%
  \scaleto{%
    \scalerel*[\widthof{\ensuremath{#1}}]{\kern-.6pt\bigwedge\kern-.6pt}%
    {\rule[-\textheight/2]{1ex}{\textheight}}
  }{\textheight}%
}{0.5ex}}%
\stackon[1pt]{#1}{\tmpbox}%
}
\title{Limit formulas for the trace of the functional calculus of quantum channels for $SU(2)$}
\author{Robin van Haastrecht}
\date{}

\newcommand{\nm}[1]{ || #1 || }
\newcommand{\C}{\mathbb{C}}
\newcommand{\R}{\mathbb{R}}
\newcommand{\N}{\mathbb{N}}
\newcommand{\Hc}{\mathcal{H}}
\newcommand{\li}[1]{\overline{ #1}}
\newcommand{\Tr}{\mathrm{Tr}}

\begin{document}

\maketitle

\begin{abstract}
Lieb and Solovej \cite{liebsolBloch} studied traces of quantum channels, defined by the leading component in the decomposition of the tensor product of two irreducible representations of $SU(2)$, to establish a Wehrl-type inequality for integrals of convex functions of matrix coefficients. It is proved that the integral is the limit of the trace of the functional calculus of quantum channels. In this paper, we introduce new quantum channels for all the components in the tensor product and generalize their limit formula. We prove that the limit can be expressed using Berezin transforms.
\end{abstract}

\section{Introduction}

The study of estimates of matrix coefficients of unitary representations of Lie groups is of fundamental interest. The case of Schrödinger representations of the Heisenberg group $\C^n \rtimes \R$ has been studied extensively, see \cite{lieb, wehrl}. Later in \cite{liebsolBloch} Lieb and Solovej proved certain Wehrl-type $L^2$-$L^p$-estimates for matrix coefficients of $SU(2)$ representations. To prove this they introduce quantum channels. The two main ingredients of their proof are firstly finding inequalities on partial sums of eigenvalues of the quantum channels, and secondly a limit of the trace of quantum channels. More precisely, let $\Hc_{\mu}$ be the irreducible $(\mu + 1)$-dimensional representation of $SU(2)$ and consider the tensor product decomposition of two irreducible representations of $SU(2)$ \cite[appendix C]{hall}

\begin{equation}
\label{decompintro}
\Hc_{\mu} \otimes \Hc_{\nu} \cong \bigoplus_{k=0}^{\mu} \Hc_{\mu + \nu - 2k}.
\end{equation}
Lieb and Solovej \cite{liebsolBloch} define:
$$ \mathcal{T}^{\nu}(A) = P (A \otimes I_{\nu}) P^* \in B(\Hc_{\mu + \nu}).$$
Here $I_{\nu}$ is the identity operator on $\Hc_{\nu}$ and
$$P: \Hc_{\mu} \otimes \Hc_{\nu} \rightarrow \Hc_{\mu + \nu}$$
is a partial isometry. The map $\mathcal{T}^{\nu}$ is
trace-preserving up to a constant and thus a quantum channel up to this constant. Roughly speaking, it is proved that the integral of functions of matrix coefficients
$$ \int_{SU(2)} \phi( |\langle g \cdot u, v_{\mu} \rangle|^2|) dg $$
is the limit of the trace of the functional calculus for $A = u \otimes u^*$ \cite{liebsolBloch}. Here $v_{\mu}$ is a highest weight vector.

In this paper we introduce general quantum channels, defined by projecting onto the irreducible $k$ component of our decomposition (\ref{decompintro}). We define

$$ \mathcal{T}_{\mu,k}^{\nu}(A) = P_k ( A \otimes I_{\nu} ) P_k^*.$$
Again, the map is trace-preserving up to a constant. It is automatically completely positive. We will study the limit formula of the trace of the functional calculus. It will turn out that the Berezin transform and the Toeplitz operator, which is equal to $(\mu + 1) R_{\mu}^*$, where $R_{\mu}$ is the symbol defined in definition \ref{symboldef}, will be useful to study the limit. We obtain the following Theorem.

\begin{theorem}
Let $\phi \in C([0,1])$. Then
$$ \lim_{\nu \rightarrow \infty} \frac{1}{\nu} \Tr(\phi(\mathcal{T}_{\mu,k}^{\nu}(R_{\mu}^*(f)))) = \int_{\C} \phi(E_{\mu,k}(f)) d \iota(z),$$
for any $f \in C(\mathbb{CP}^1)$ such that $\int_{\mathbb{CP}^1} f(z) dz = 1$ and with Toeplitz operator $R_{\mu}^*(f) \geq 0$.
\end{theorem}
Here we have that:
$$ E_{\mu,k}(f) = \binom{\mu}{k} \sum_{l=0}^k (-1)^{k-l} \binom{k}{l} B_{\mu - l}(f).$$
Using denseness of polynomials in continuous functions on a compact interval we extend the result to the functional calculus of $A$ of continuous functions. We also note that $R_{\mu}^*$ is surjective, so it is enough to consider $\mathcal{T}_{\mu,k}^{\nu}(R^*_{\mu}(f))$ instead of $\mathcal{T}_{\mu,k}^{\nu}(A)$. We also study the operator $E_{\mu}$ and find the eigenvalues.

The method we use it the following. Realizing all transforms as integral kernels, we calculate:
$$R_{\nu + \mu - 2k} \mathcal{T}_{\mu,k}^{\nu}(R_{\mu}^*(f)).$$
Now we can calculate the trace of our operator
$$\mathcal{T}_{\mu,k}^{\nu} R_{\mu}^* = R_{\mu + \nu - 2k}^* B_{\mu + \nu - 2k}^{-1} R_{\mu + \nu - 2k} \mathcal{T}_{\mu,k}^{\nu} R_{\mu}^* $$
in limit, as the inverse Berezin transform $(\mu + \nu - 2k - 1) B_{\mu + \nu - 2k}^{-1}$ will be going to the identity strongly, $(\nu + 1)^n R_{\nu}( R_{\nu}^*(f)^n)$ is going to $f^n$ and we have the formula
$$ \Tr(A) = (\nu + 1) \int_{\C} R_{\mu}(A)(z) d \iota(z).$$
Now we use denseness of polynomials in $C([0,1])$.

Lieb and Solovej \cite{liebsolSymm} studied similar questions for the compact group $SU(n)$, proving similar inequalities by considering tensor products $\bigodot^{\mu} \C^n$ and projecting onto the leading components, and for the non-compact group $SU(1,1)$, by projecting onto the lowest weight component, also called the Cartan component, in the tensor product of two highest weight representations \cite{liebsolSU11}. In this case they proved some Wehrl-type inequalites; see also \cite{frank, sugi, kuli}. Some inequalities were improved by Frank \cite{frank}. It seems that these questions can be put in a more general context of highest weight representations of Hermitian Lie groups \cite{zhangCon}. In a future work, I will consider the unitary highest weight representations of the non-compact group $SU(1,1)$ realized on the weighted Bergman spaces and study the corresponding problem. I will calculate the trace of the quantum channels for all components in the tensor product of unitary highest weight representation of the non-compact group $SU(1,1)$, which is closely related to the case $SU(2)$.

We note that the Berezin transform is closely related to quantization on Kähler manifolds in Geometry and Mathematical Physics and have been studied extensively; see e.g. \cite{alieng, BMSCMP, UU}. Some of our results about Berezin transforms might be obtained from these results. However, we provide more precise results using the representation of $SU(2)$.

The paper is organized as follows. First we introduce representations of $SU(2)$ as reproducing kernel Hilbert spaces, go through some general theory and introduce quantum channels in Section \ref{backinfo}. In Section \ref{branproc} we realize the elements of $B(\Hc_{\mu})$ using the operator $R_{\mu}^*$, a Toeplitz operator, and we calculate $R_{\nu + \mu - 2k} \mathcal{T}_{\mu,k}^{\nu}(R_{\mu}^*(f))$. Finally, in Section \ref{funccalc} we study the Berezin transform $B_{\nu}$ and calculate $\lim_{\nu \rightarrow \infty} \Tr(\mathcal{T}_{\mu,k}^{\nu} R_{\mu}^*)$. We also study the operator $E_{\mu,k}$.

\section*{Acknowledgements}
I want to thank Genkai Zhang for fruitful and inspiring discussions.

\section{Preliminaries}
\label{backinfo}
Notation: $(a)_n = a(a+1) \dots (a+n-1)$ is the rising Pochhammer symbol and $(a)^n_{-} = a(a-1) \dots (a-n+1)$ is the falling Pochhammer symbol.

We study the representation theory of $SU(2)$. We define $\Hc_{\nu}$.

\begin{definition}
\label{innerprodspace}
Let $\Hc_{\nu}$ be the space of polynomials on $\C$ in $z$ of degree less than or equal to $\nu$ and let the inner product on it be given by

$$\langle f, g \rangle = \int_{ \C } f(z) \li{g(z)} d \iota_{\nu}(z),$$
where $d \iota_{\nu}(z) = \frac{(\nu + 1)}{(1 + |z|^2)^{ \nu }} \frac{dz}{\pi (1 + |z|^2)^{2}}$.
\end{definition}

We note that $d \iota_{0}(z) = \frac{dz}{ \pi (1 + |z|^2)^{2}}$ is the $SU(2)$-invariant measure $d \iota$ on $\C$ such that $\iota(\C)=  1$ (as a coordinate chart in $\mathbb{CP}^1$). Also, the norm is normalized such that $\nm{1}_{\nu} = 1$. The set $\{z^i\}_{i=0}^{\nu}$ is an orthogonal basis for $\Hc_{\nu}$ where
$$\nm{z^i}_{\nu}^2 = \binom{\nu}{i},$$
and $\Hc_{\nu}$ is a reproducing kernel Hilbert space (RKHS for short) with kernel $$K^{\nu}(z,w) = (1 + z \li{w})^{\nu}.$$

\begin{remark}
\label{repsu2facts}
Note that $SU(2)$ acts on the symmetric space $SU(2) / U(1)$ by left multiplication. This space is isomorphic to $\C \cup \{ \infty \} = \mathbb{CP}^1$ when the action on $\mathbb{CP}^1$ is given by $g \cdot [z_1 : z_2] = [g(z_1 : z_2)]$, matrix multiplication. Then we see that $U(1)$ is exactly the subgroup fixing $[0:1]$, and $SU(2)$ is transitive as it is transitive on the unit ball. Thus $\mathbb{CP}^1 \cong SU(2) / U(1)$. We note that $\C \subseteq \mathbb{CP}^1$ by $z \mapsto [z:1]$ is a dense coordinate chart. Thus it is enough to consider $\C \subseteq \mathbb{CP}^1$ and the action of $SU(2)$ on $G$ is given by:
$$\begin{pmatrix} a & b \\ c & d \end{pmatrix} \cdot z = \frac{a z + b}{c z + d}.$$
We shall mostly work on $\C$. We note that there is a $SU(2)$-invariant metric on $\mathbb{CP}^1$ making it a Riemannian symmetric space, see \cite[chapter VII, Proposition 1.1]{helg}.
\end{remark}
We now define the $SU(2)$-representation on $\Hc_{\nu}$.

\begin{definition}
Let $g \in SU(2)$ such that $g^{-1} = \begin{pmatrix} a & b \\ c & d \end{pmatrix} = \begin{pmatrix} a & b \\ - \overline{b} & \overline{a} \end{pmatrix}$, where $a,b \in \C$ such that $|a|^2 + |b|^2 = 1$. The representation is given by
$$ g \cdot f (z) = (-\li{b} z + \li{a})^{\nu} f(g^{-1} \cdot z) = (-\li{b} z + \li{a})^{\nu} f(\frac{az + b}{- \overline{b} z + \li{a}}).$$
\end{definition}

This is a unitary irreducible representation, the unique representation of $SU(2)$ of dimension $\nu + 1$. It is isomorphic to the usual realization of homogeneous polynomials of degree $\nu$.

Now let $\mu, \nu \in \N$ and $\nu > \mu$ (in the end we will let $\nu \rightarrow \infty$), and it is well known that \cite[appendix C]{hall}:

\begin{equation}
\label{tensorprodmult}
\Hc_{\mu} \otimes \Hc_{\nu} \cong \bigoplus_{k = 0}^{\mu} \Hc_{\mu + \nu - 2k}.
\end{equation}

We define a map $J_k: \Hc_{\mu} \otimes \Hc_{\nu} \rightarrow \Hc_{\mu + \nu - 2k}$.

\begin{definition}
Let $J_k: \Hc_{\mu} \otimes \Hc_{\nu} \rightarrow \Hc_{\mu + \nu - 2k}$ be given by
\begin{flalign*}
J_k(f)(\xi) & = \int_{\C^2} f(z,w) ( \frac{\li{z}}{1 + \xi \li{z}} - \frac{\li{w}}{1 + \xi \li{w}} )^k (1 + \xi \li{z})^{\mu} (1 + \xi \li{w})^{\nu} d \iota_{\mu}(z) d \iota_{\nu}(w)
\\ & = \int_{\C^2} f(z,w) (\li{z} - \li{w})^k (1 + \xi \li{z})^{\mu-k} (1 + \xi \li{w})^{\nu-k} d \iota_{\mu} (z) d \iota_{\nu}(w).
\end{flalign*}
\end{definition}

This operator $J_k$ is also a differential operator given by
$$J_k = (-1)^k \sum_{j = 0}^k (-1)^j \binom{k}{j} \frac{1}{ (- \mu)_j (- \nu)_{k-j}} \partial^j_z \partial^{k-j}_w f |_{z = w = \xi},$$
and it intertwines the representations. This can be proved using the reproducing kernel formula, but we skip the details here. The counterpart of $J_k$ for the holomorphic discrete series representations of $SL(2,\R)$ is called the Rankin-Cohen operator, see for example \cite[(1), p. 58]{zag}. The adjoint of $J_k$ can be obtained by direct computations.

\begin{lemma}
The adjoint of $J_k$ is:
$$ J_k^*(f)(z,w) = \int_{\C} (1 + z \li{\xi})^{\mu - k}(1 + w \li{\xi})^{\nu - k} f(\xi) (z-w)^k d \mu_{\mu + \nu - 2k}(\xi).$$
\end{lemma}

The map $J_k J_k^*: \Hc_{\mu + \nu - 2k} \rightarrow \Hc_{\mu + \nu - 2k}$ is a scalar multiple of the identity by Schur's Lemma, since $\Hc_{\mu + \nu - 2k}$ is irreducible. To find the scalar we calculate $J_k J_k^*(f)(0)$ for $f = 1$.

\begin{proposition}
\label{intertwinerconstant}
We have
$$ C_{\mu,\nu,k}^{-2} := J_k (J^*_k(1))(0) = \frac{k! (\mu + \nu - 2k + 2)_k}{(-\nu)_k (-\mu)_k}.$$
\end{proposition}

Before we prove this we recall the following summation formula \cite[Corollary 2.2.3]{andaskroy}.

\begin{lemma}
\label{Gausshgf}
For $n$ a non-negative integers and $b,c$ integers such that $|c| \geq n$ we have that
$${}_2F_1(-n, b, c, 1) = \frac{(c-b)_n}{(c)_n}$$
\end{lemma}

Now we are ready to prove Proposition \ref{intertwinerconstant}.

\begin{proof}
First we see that

\begin{flalign*}
J^*_k(1)(z,w) & = \int_{\C} (1 + z \li{\xi})^{\mu - k}(1 + w \li{\xi})^{\nu - k} (z-w)^k d \mu_{\mu + \nu - 2k}(\xi)
\\ & = (z-w)^k \sum_{j = 1}^{\mu - k} \sum_{i = 1}^{\nu - k} \binom{\mu - k}{j} \binom{\nu - k}{i} \int_{\C} z^j \li{\xi}^j w^i \li{\xi}^i d \mu_{\mu + \nu - 2k}(\xi)
\\ & = (z-w)^k \sum_{j = 1}^{\mu - k} \sum_{i = 1}^{\nu - k} \binom{\mu - k}{j} \binom{\nu - k}{i} w^i z^j \int_{\C} \li{\xi}^j \li{\xi}^i d \mu_{\mu + \nu - 2k}(\xi)
\\ & = (z-w)^k \sum_{j = 1}^{\mu - k} \sum_{i = 1}^{\nu - k} \binom{\mu - k}{j} \binom{\nu - k}{i} w^i z^j \langle 1, \xi^{i + j} \rangle
\\ & = (z-w)^k.
\end{flalign*}
Then we see that

\begin{flalign*}
J_k (J^*_k(1))(0) & = J_k( (z-w)^k)(0)
\\ & = \int_{\C^2} (z-w)^k (\li{z} - \li{w})^k 1^{\mu} 1^{\nu} d \iota_{\mu}(z) d \iota_{\nu}(w)
\\ & = \int_{\C^2} (z-w)^k (\li{z} - \li{w})^k d \iota_{\mu}(z) d \iota_{\nu}(w)
\\ & = \nm{ (z-w)^k}_{\Hc_{\mu} \otimes \Hc_{\nu}}^2
\\ & = \nm{ \sum_{j=0}^k (-1)^j \binom{k}{j} z^j w^{k-j} }_{\Hc_{\mu} \otimes \Hc_{\nu}}^2
\\ & = \sum_{j=0}^k \binom{k}{j}^2 \nm{z^j}_{\mu}^2 \nm{w^{k-j}}^2_{\nu}
\\ & = \sum_{j=0}^k \binom{k}{j}^2 \binom{\mu}{j}^{-1} \binom{\nu}{k-j}^{-1}
\\ & = k! \sum_{j=0}^k \binom{k}{j} \frac{1}{(\mu)^j_{-} (\nu)^{k-j}}_{-}
\\ & = \frac{k!}{(\nu+1-k)_k} \sum_{j=0}^k \binom{k}{j} \frac{(-1)^j (\nu+1 -k)_j}{(- \mu)_j}
\\ & = \frac{k!}{(\nu + 1 -k)_k} {}_2F_{1}(-k, \nu+1-k, -\mu, 1).
\end{flalign*}
Using Lemma \ref{Gausshgf} we get

\begin{flalign*}
J_k (J^*_k(1))(0) & = \frac{k!}{(\nu + 1 -k)_k} {}_2F_{1}(-k, \nu+1-k, -\mu, 1)
\\ & = \frac{k!}{(\nu + 1 -k)_k} \frac{(- \mu - \nu - 1 + k)_k}{(-\mu)_k}
\\ & = \frac{k!}{(\nu)^k} \frac{(\mu+\nu+1-k)^k}{(\mu)^k}
\\ & = \frac{k! (\mu + \nu - 2k + 2)_k}{(\nu)^k (\mu)^k} = \frac{k! (\mu + \nu - 2k + 2)_k}{(-\nu)_k (-\mu)_k}.
\end{flalign*}
\end{proof}

We define the projection onto the irreducible subspaces.

\begin{definition}
We define $P_k := C_{\mu,\nu,k} J_k: \Hc_{\mu} \otimes \Hc_{\nu} \rightarrow \Hc_{\mu + \nu - 2k}$, where $C_{\mu, \nu,k}$ is defined in Proposition \ref{intertwinerconstant}.
\end{definition}

We see that $P_k$ is the orthogonal projection onto the subspace $\Hc_{\mu + \nu - 2k}$ in Decomposition (\ref{tensorprodmult}). We need also the decomposition of $L^2(SU(2)/U(1)) = L^2(\mathbb{CP}^1)$ under the action of $SU(2)$
\begin{equation}
\label{multfree}
L^2(\mathbb{CP}^1) = \bigoplus_{k=0}^{\infty} \widetilde{\Hc_{2k}},
\end{equation}
where the sum is in the $L^2$ sense and $\Hc_{2k} \cong \widetilde{\Hc_{2k}}$ by the map
$$ v \mapsto (g \mapsto \langle v, g \cdot v_0 \rangle),$$
where $v_0$ is the $U(1)$-fixed vector in $\Hc_{2k}$, i.e. $f(z) = z^k$. Note that this decomposition is multiplicity free. For a more general result, see \cite[chapter V, Theorem 4.3]{helggga}. Also, $C^{\infty}(\mathbb{CP}^1) \subseteq L^2(\mathbb{CP}^1)$ is dense and any $f \in C^{\infty}(\mathbb{CP}^1)$ can be written as
$$ f = \sum_{k=0}^{\infty} f_k,$$
where $f_k \in \Hc_{2k}$ and the convergence is absolute convergence.

Now we define some further $SU(2)$-invariant operators on the representation spaces. First we define the map $R_{\nu}$, which is commonly called the symbol of the operator.

\begin{definition}
\label{symboldef}
Let $R_{\nu}$ be the map
$$R_{\nu} : \Hc_{\nu} \otimes \li{\Hc_{\nu}} \rightarrow C^{\infty}( \mathbb{CP}^1) \subseteq L^2(\mathbb{CP}^1),$$
given by $R_{\nu}(f_1 \otimes \li{f_2}) = \frac{f_1(z) \li{f_2(z)}}{(1 + |z|^2)^{\nu}}$, for $f_1, f_2 \in \Hc_{\nu}$.
\end{definition}

Observe that $\li{f_2}$ is antiholomorphic and that the map is $SU(2)$-invariant. Note that we have the identification $\Hc_{\nu} \otimes \li{\Hc_{\nu}} = B(\Hc_{\nu})$, where $f_1 \otimes \li{f_2}$ corresponds to the kernel $(x,y) \mapsto f_1(x)\li{f_2(y)}$ of an operator. Note that by Equation (\ref{tensorprodmult})
$$ \Hc_{\nu} \otimes \li{\Hc_{\nu}} \cong \bigoplus_{i=0}^{\nu} \Hc_{2 i}.$$
A direct calculation shows that $R_{\nu}^*(f) = \frac{1}{\nu + 1} T_f$, where $T_f$ is the Toeplitz operator given by $T_f(g) = P M_f P^*(h)$, where $h \in L^2(\mathbb{CP}^1)$, $M_f$ is multiplication by $f$ and $P$ is projecting onto $\Hc_{\nu}$ from $L^2(\mathbb{CP}^1)$. Explicitly, this means

$$ T_f(h)(z) = \int_{\C} f(x)h(x) K^{\nu}(z,x) d \iota_{\nu}(x).$$
Its kernel is $T_f(x,y) = \int_{\C} K^{\nu}(x,z) K^{\nu}(z,y) f(z) d \iota_{\nu}(z)$. Hence we get

$$R_{\nu}^*(f)(x,y) = \frac{1}{\nu + 1} \int_{\C} K^{\nu}(x,z)K^{\nu}(z,y)f(z) d \iota_{\nu}(z).$$

\begin{remark}
\label{Rnuinjective}
The operator $R_{\nu}$ is injective by general arguments; it can also be proved by differentiating by $\partial_z$ and $\partial_{\li{z}}$. As a consequence, the map $R_{\nu}^*$ is surjective as the $\Hc_{\nu}$ are finite dimensional, and the maps

$$R_{\nu}: \Hc_{\nu} \otimes \li{\Hc_{\nu}} \rightarrow \mathrm{Im}(R_{\nu})$$
and
$$R_{\nu}^*: \mathrm{Im}(R_{\nu}) \rightarrow \Hc_{\nu} \otimes \li{\Hc_{\nu}}$$
are bijective.
\end{remark}

We now define the Berezin transform.

\begin{definition}
\label{berdef}
The Berezin transform $B_{\nu}$ is defined by
$$B_{\nu} := R_{\nu} R_{\nu}^*: L^2(\mathbb{CP}^1) \rightarrow L^2(\mathbb{CP}^1).$$
\end{definition}

Note that $B_{\nu}(C^{\infty}(\mathbb{CP}^1)) \subseteq C^{\infty}(\mathbb{CP}^1)$. More explicitly, for $f \in C(\mathbb{CP}^1)$ we see that

$$ B_{\nu}(f)(z) = \int_{\C} \frac{(1 + z \li{s})^{\nu} (1 + \li{z} s)^{\nu}}{(1 + |z|^2)^{\nu} (1 + |s|^2)^{\nu}} f(s) d \iota (s).$$
Note that by Cauchy-Schwarz
$$ |B_{\nu}(f)(z)| \leq \int_{\C} |\frac{(1 + z \li{s})^{\nu} (1 + \li{z} s)^{\nu}}{(1 + |z|^2)^{\nu} (1 + |s|^2)^{\nu}}| \cdot |f(s)| d \iota (s) \leq \nm{f}_{\infty}.$$
Observe that the Berezin transform is $SU(2)$-invariant. Now by Decomposition \ref{multfree} and Schur's lemma, $B_{\nu}$ must act as a constant on each subspace $\Hc_{2k} \subseteq C^{\infty}(\mathbb{CP}^1)$. From \cite{zhangBer} we find that for $f \in \widetilde{\Hc_{2k}}$
\begin{equation}
\label{Bnuconstant}
B_{\nu}(f) =  \frac{(\nu!)^2}{(\nu + k + 1)! \cdot (\nu - k)!} f
\end{equation}
when $k \leq \nu$, and
\begin{equation}
\label{Bnu0}
B_{\nu}(f) = 0
\end{equation}
when $k > \nu$.

As we know that any operator $A \in B(\Hc_{\nu})$ is of the form $R_{\nu}^*(f)$ for some $f$ in $\mathrm{Im}(R_{\nu})$, a natural question is to see which $A \in \Hc_{\nu} \otimes \li{\Hc_{\nu}} = B(\Hc_{\nu})$ corresponds to which $f \in \mathrm{Im}(R_{\nu})$; that is, we want to find $(R_{\nu}^*)^{-1}(A) = B_{\nu}^{-1} R_{\nu}(A)$. We first split $A$ into the $\widetilde{\Hc_{2k}}$ components, $0 \leq k \leq \nu$, i.e. $A = A_0 + \dots + A_{\nu}$, where $A_{k} \in \widetilde{\Hc_{2 k}} \subseteq \Hc_{\nu} \otimes \li{\Hc_{\nu}}$. As $B_{\nu}$ acts as the constant $\frac{(\nu!)^2}{(\nu + k + 1)! \cdot (\nu - k)!}$ on the space $\Hc_{2k}$, $B_{\nu}^{-1}$ acts as the constant $C_{\nu}(k) = \frac{(\nu + k + 1)! \cdot (\nu - k)!}{(\nu!)^2}$ on that space. We conclude that

$$ (R_{\nu}^*)^{-1}(A)(z) = \sum_{k=0}^{\nu} C_{\mu}(k) \frac{A_k(z,z)}{(1 + |z|^2)^{\nu}}.$$

We recall some simple facts about reproducing kernel Hilbert spaces. For a (finite dimensional) reproducing kernel Hilbert spaces $\Hc$ of functions on a measure space $(X, \mu)$, equipped with the $L^2(X)$ inner product and reproducing kernel $K(x,y)$, we have for any operator $\Gamma$ on it, writing $K_x(y) = K(y,x)$
$$ \Gamma(f)(x) = \int_{X} L(x,y) f(y) d \iota_{\nu}(y),$$
where $x \mapsto L(x,y)$ is in $\Hc$ and $y \mapsto \li{L(x,y)}$ are in $\Hc$. We note that
$$L(x,y) = \li{\Gamma^*(K_x)(y)} = \langle K_y, \Gamma^*(K_x) \rangle = \langle \Gamma(K_y), K_x \rangle = \Gamma(K_y)(x),$$
Now let $\{e_i\}_{i=1}^{\dim(\Hc)}$ be an orthonormal base for $\Hc$, then
$$K(x,y) = \sum_{i=1}^{\dim(\Hc)} e_i(x) \li{e_i(y)},$$
and for $\Gamma$ an operator on $\Hc_{\mu}$
$$ \Tr(\Gamma) = \int_{X} L(y,y) d \iota_{\nu}(y).$$

\section{Quantum channels \& Berezin transforms}
\label{branproc}

We now define quantum channels.

\begin{definition}
Associated to Decomposition (\ref{tensorprodmult}) we define
$$\mathcal{T}_{\mu,k}^{\nu}: B(\Hc_{\mu}) \rightarrow B(\Hc_{\mu + \nu - 2k}),$$
given by $\mathcal{T}_{\mu,k}^{\nu}(A) =  P_k \circ A \otimes I \circ P_k^*$. Furthermore, we define a renormalization
$$\widehat{\mathcal{T}}_{\mu,k}^{\nu}(A) = \frac{\mu + 1}{\mu + \nu - 2k + 1} P_k \circ A \otimes I \circ P_k^*.$$
\end{definition}

These maps will turn out to be positive and $\widehat{\mathcal{T}}_{\mu,k}^{\nu}$ will be trace-preserving, making it a quantum channel. These channels were defined by Lieb and Solovej \cite{liebsolBloch} for the case $k=0$. To my knowledge, the quantum channels for general $k$ are first studied here. There are many questions one can ask of these channels, similar to the case $k = 0$ studied by Lieb and Solovej. Here we will mainly study limit formulas of the trace. We prove some properties of the maps $\mathcal{T}_{\mu,k}^{\nu}$.

\begin{proposition}
The map $\mathcal{T}_{\mu,k}^{\nu}$ is completely positive, and the map $\widehat{\mathcal{T}}_{\mu,k}^{\nu}$ is completely positive and trace-preserving.
\end{proposition}

\begin{proof}
It is obvious that the map $\mathcal{T}_{\mu,k}^{\nu}$ is completely positive, as it is the composition of the $*$-homomorphism $A \mapsto A \otimes I$ and $A \mapsto P_k A  P_k^*$. It follows that $\widehat{\mathcal{T}}_{\mu,k}^{\nu}$ is completely positive as well.

Now we prove $\widehat{\mathcal{T}}_{\mu,k}^{\nu}$ is trace preserving. We see
\begin{flalign*}
\Tr(\widehat{\mathcal{T}}_{\mu,k}^{\nu}(A)) & = \frac{\mu + 1}{\mu + \nu - 2k + 1} \Tr(P_k \circ A \otimes I \circ P_k^*)
\\ & = \frac{\mu + 1}{\mu + \nu - 2k + 1} \int_{SU(2)} \Tr(g P_k \circ A \otimes I \circ P_k^* g^{-1}) dg
\\ & = \frac{\mu + 1}{\mu + \nu - 2k + 1} \int_{SU(2)} \Tr(P_k \circ gAg^{-1} \otimes I \circ P_k^*) dg
\\ & = \frac{\mu + 1}{\mu + \nu - 2k + 1} \Tr(P_k \circ (\int_{SU(2)} gAg^{-1} dg) \otimes I \circ P_k^*).
\end{flalign*}
By Schur's Lemma $\int_{SU(2)} gAg^{-1} dg = \frac{
\Tr(A)}{\dim(\Hc_{\mu})} I$. Hence
$$ \Tr(\widehat{\mathcal{T}}_{\mu,k}^{\nu}(A)) = \frac{\mu + 1}{\mu + \nu - 2k + 1} \Tr(A) \frac{\dim(\Hc_{\nu + \mu - 2k})}{\dim(\Hc_{\mu})} = \Tr(A).$$
\end{proof}

As in \cite{liebsolBloch}, our eventual goal is calculating $\lim_{\nu \rightarrow \infty} \frac{1}{\dim(\Hc_{\mu + \nu - 2k})} \Tr( \mathcal{T}_{\mu,k}^{\nu}(A))$.

\begin{remark}
\label{munuktonu}
Notationally it is easier to write $\lim_{\nu \rightarrow \infty} \frac{1}{\nu} \Tr( \mathcal{T}_{\mu,k}^{\nu}(A))$ instead of $\lim_{\nu \rightarrow \infty} \frac{1}{\dim(\Hc_{\mu + \nu - 2k})} \Tr( \mathcal{T}_{\mu,k}^{\nu}(A))$. They will have the same limit, as 
$$\dim(\Hc_{\mu+\nu-2k}) = \mu + \nu - 2k + 1.$$
In the rest of the paper, we will often exchange $\mu + \nu - 2k + 1$ with $\nu$ for notational convenience.
\end{remark}

We compute the integral kernel of $\mathcal{T}_{\mu,k}^{\nu}(A)$.

\begin{proposition}
\label{gammakernel}
The operator $\mathcal{T}_{\mu,k}^{\nu}(A) = P_k \circ A \otimes I \circ P_k^*$ has kernel
\begin{flalign*}
L(\xi, x) & = P_k \circ A \otimes I \circ P_k^* (K_x)(\xi)
\\ & = C_{\mu,\nu,k}^2  \int_{\C^3} (1 + u \li{x})^{\mu - k} (1 + w \li{x})^{\nu - k} (u - w)^k A(z,u) (\li{z} - \li{w})^k \cdot
\\ & (1 + \xi \li{z})^{\mu - k} (1 + \xi \li{w})^{\nu - k} d \iota_{\mu}(u) d \iota_{\mu}(z) d \iota_{\nu}(w),
\end{flalign*}
where $A(x,y)$ is the kernel of $A$.
\end{proposition}

\begin{proof}
We calculate

\begin{flalign*}
J_k^*(K_x)(z,w) & = \int_{\C} (1 + z \li{\xi})^{\mu - k} (1 + w \li{\xi})^{\nu - k} K(\xi,x) (z - w)^k d \mu_{\mu + \nu - 2k}(\xi)
\\ & = \langle K_x, (1 + \li{z} \cdot -)^{\mu - k} (1 + \li{w} \cdot -)^{\nu - k}(\li{z} - \li{w})^k \rangle
\\ & = (1 + z \li{x})^{\mu - k} (1 + w \li{x})^{\nu - k}(z - w)^k.
\end{flalign*}

Then we see that

\begin{flalign*}
(A \otimes I) \circ J_k^*(K_x)(z,w) & = \int_{\C} J_k^*(K_x)(u,w) A(z,u) d \iota_{\mu}(u)
\\ & = \int_{\C} (1 + u \li{x})^{\mu - k} (1 + w \li{x})^{\nu - k} (u-w)^k A(z,u) d \iota_{\mu}(u).
\end{flalign*}

Thus

\begin{flalign*}
J_k \circ (A \otimes I) \circ J_k^* (K_x) (\xi) & = \int_{\C^2} (A \otimes I) \circ J_k^* (K_x) (z,w) \cdot
\\ & (\frac{ \li{z} }{1 + \xi \li{z}} - \frac{\li{w}}{1 + \xi \li{w}})^k (1 + \xi \li{z})^{\mu} (1 + \xi \li{w})^{\nu} d \iota_{\mu}(z) d \iota_{\nu}(w)
\\ & = \int_{\C^3} (1 + u \li{x})^{\mu - k} (1 + w \li{x})^{\nu - k} (u - w)^k A(z,u) \cdot
\\ & (\li{z} - \li{w})^k (1 + \xi \li{z})^{\mu - k} (1 + \xi \li{w})^{\nu - k} d \iota_{\mu}(u) d \iota_{\mu}(z) d \iota_{\nu}(w).
\end{flalign*}
Combining this with $P_k = C_{\mu,\nu,k} J_k$ we get the result.
\end{proof}

We consider the case $k = 0$. Then the formula simplifies to
$$P_0(A \otimes I)P_0^* (K_x)(\xi)  = A(\xi,x) (1 + \li{x} \xi)^{\nu}.$$
In the case where $k$ is maximal, $k = \mu$, our kernel simplifies to

\begin{flalign*}
& L(\xi, x) = \frac{\nu + 1}{\nu - \mu + 1} \int_{\C^3} (1 + w \li{x})^{\nu - \mu} (u - w)^{\mu} \li{X(u)} X(z) \cdot
\\ & (\li{z} - \li{w})^{\mu} (1 + \xi \li{w})^{\nu - \mu} d \iota_{\mu}(u) d \iota_{\mu}(z) d \iota_{\nu}(w).
\end{flalign*}
Then we note that $(u - w)^{\mu} = (-w)^{\mu} (1 - \frac{u}{w})^{\mu}$, thus by the reproducing kernel property the above becomes

$$L(\xi, x) = \frac{\nu + 1}{\nu - \mu + 1} \int_{\C} | (g \cdot X)(w)|^2 (1 + w \li{x})^{\nu - \mu} (1 + \xi \li{w})^{\nu - \mu} d \iota_{\nu}(w),$$
where $g = \begin{pmatrix} 0 & 1 \\ -1 & 0 \end{pmatrix}$. We want to calculate the trace of the functional calculus,
$$\frac{1}{\nu} \Tr(\phi(\mathcal{T}_{\mu,k}^{\nu}(A)))$$
for a continuous function $\phi$. We start with $\phi(x) = x^n $, where $n = 1,2,\dots$, i.e. the trace 
$$\frac{1}{\nu} \Tr(\mathcal{T}_{\mu,k}^{\nu}(A)^n)$$
for any $n$. As $R_{\mu}^*$ is surjective, it is enough to calculate
$$ \lim_{\nu \rightarrow \infty} \frac{1}{\nu} \Tr(\phi(\mathcal{T}_{\mu,k}^{\nu}(R_{\mu}^*(f)))).$$
It will turn out this, and in particular $R_{\mu + \nu - 2k} \mathcal{T}_{\mu,k}^{\nu} R_{\mu}^*(f)$, will be easier to handle.

\begin{proposition}
We let
$$E_{\mu,k}^{\nu}(f) := R_{\mu + \nu - 2k} \mathcal{T}_{\mu, k}^{\nu} R_{\mu}(f) .$$
Then
\begin{flalign*}
& E_{\mu,k}^{\nu}(f) = C_{\mu,\nu,k}^2 \sum_{l = 0}^k (-1)^{k-l} \binom{\nu - k}{k-l} \frac{(\nu - k + l)! k !}{\nu! l!} B_{\mu - l}(f).
\end{flalign*}
\end{proposition}

\begin{remark}
We make a general remark on $SU(2)$-invariance. We see directly that $R_{\mu}$, $R_{\mu}^*$ and $\mathcal{T}_{\mu,k}^{\nu}$ are $SU(2)$-invariant. Hence it will often be enough to ascertain an identity in $0$ to find it for the whole space as $SU(2)$ acts transitively on $\mathbb{CP}^1$. The following proof will be an example of that. We also know that any $\Hc_{\mu} \otimes \Hc_{\nu}$ and $L^2(\mathbb{CP}^1)$ split multiplicity free by Equations (\ref{tensorprodmult}) and (\ref{multfree}). As $R_{\mu}$, $R_{\mu}^*$ and $\mathcal{T}_{\mu,k}^{\nu}$ are $SU(2)$-invariant, Schur's Lemma says they act as scalars on the irreducible components.
\end{remark}

\begin{proof}
We first prove
\begin{flalign*}
& E_{\mu,k}^{\nu}(f)(0) = C_{\mu,\nu,k}^2 \sum_{l = 0}^k (-1)^{k-l} \binom{\nu - k}{k-l} \frac{(\nu - k + l)! k !}{\nu! l!} B_{\mu - l}(f)(0).
\end{flalign*}
Then by $SU(2)$-invariance we get by the above remark that

$$E_{\mu,k}^{\nu}(f) = C_{\mu,\nu,k}^2 \sum_{l = 0}^k (-1)^{k-l} \binom{\nu - k}{k-l} \frac{(\nu - k + l)! k !}{\nu! l!} B_{\mu - l}(f).$$
By Proposition \ref{gammakernel} for any operator $A$ on $\Hc_{\mu}$ we have

\begin{flalign*}
& (\mu + \nu - 2k + 1) \mathcal{T}_{\mu,k}^{\nu}(A)(x,y) = C_{\mu,\nu,k}^2 \int_{\C^3} (1 + u \li{y})^{\mu - k} (1 + w \li{y})^{\nu - k} \\ & (u - w)^k A(z,u) (\li{z} - \li{w})^k (1 + x \li{z})^{\mu - k} (1 + x \li{w})^{\nu - k} d \iota_{\mu}(u) d \iota_{\mu}(z) d \iota_{\nu}(w),
\end{flalign*}

Now we calculate $\mathcal{T}_{\mu,k}^{\nu} R_{\mu}^*(f)$. We get

\begin{flalign*}
& \mathcal{T}_{\mu,k}^{\nu} R_{\mu}^*(f)(x,y) = \frac{C_{\mu,\nu,k}^2}{\mu + 1} \int_{\C^4} (1 + u \li{y})^{\mu - k} (1 + w \li{y})^{\nu - k} (u - w)^k K(z,s)K(s,u) \cdot
\\ & (\li{z} - \li{w})^k (1 + x \li{z})^{\mu - k} (1 + x \li{w})^{\nu - k} f(s) d \iota_{\mu}(u) d \iota_{\mu}(z) d \iota_{\nu}(w) d \iota_{\mu}(s)
\\ & = \frac{C_{\mu,\nu,k}^2}{\mu + 1} \int_{\C^2} (1 + w \li{y})^{\nu - k} (1 + s \li{y})^{\mu - k} (s - w)^k  (\li{s} - \li{w})^k \cdot
\\ & (1 + x \li{s})^{\mu - k} (1 + x \li{w})^{\nu - k} f(s) d \iota_{\nu}(w) d \iota_{\mu}(s).
\end{flalign*}
Applying $R_{\mu + \nu - 2k}$ gives us

\begin{flalign*}
& E_{\mu,k}^{\nu}(f)(z) = R_{\mu + \nu - 2k} \mathcal{T}_{\mu,k}^{\nu} R_{\mu}^*(f)(z) = \frac{C_{\mu,\nu,k}^2}{(\mu + 1)(1 + |z|)^{\mu + \nu - 2k}} \cdot
\\ & \int_{\C^2} |(1 + w \li{z})^{\nu - k} (1 + s \li{z})^{\mu - k} (s - w)^{k}|^2 f(s) d \iota_{\nu} (w) d \iota_{\mu}(s).
\end{flalign*}

We will now try to write $E_{\mu,k}^{\nu}$ as a sum of Berezin transforms. From the definition we see

\begin{flalign*}
& B_{\mu}(f)(0) = \int_{\C} \frac{f(s)}{(1 + |s|^2)^{\mu}} d \iota(s).
\end{flalign*}
By the orthogonality of the $\{ w^i \}$ we can evaluate $E_{\mu,k}^{\nu}(f)(0)$ as

\begin{flalign*}
E_{\mu,k}^{\nu}(f)(0) & = \frac{C_{\mu,\nu,k}^2}{\mu + 1} \int_{\C} |(s - w)^k|^2 f(s) d \iota_{\nu}(w) d \iota_{\mu}(s)
\\ & = C_{\mu,\nu,k}^2 \sum_{i=0}^{k} \binom{k}{i}^2 \binom{\nu}{i}^{-1} \int_{\C} \frac{|s^{k-i}|^2 f(s)}{(1 + |s|^2)^{\mu}} d \iota (s)
\\ & = C_{\mu,\nu,k}^2 \sum_{i=0}^{k} \binom{k}{i}^2 \binom{\nu}{k-i}^{-1} \int_{\C} \frac{|s^{i}|^2 f(s)}{(1 + |s|^2)^{\mu}} d \iota (s)
\\ & = C_{\mu,\nu,k}^2 \sum_{i=0}^{k} \binom{k}{i}^2 \binom{\nu}{k-i}^{-1} \int_{\C} \frac{(|s^2| + 1 - 1)^i f(s)}{(1 + |s|^2)^{\mu}} d \iota (s)
\\ & = C_{\mu,\nu,k}^2 \sum_{i=0}^{k} \binom{k}{i}^2 \binom{\nu}{k-i}^{-1} \sum_{l = 0}^i \binom{i}{l} (-1)^{i-l} \cdot
\\ & \int_{\C} \frac{(|s^2| + 1)^l f(s)}{(1 + |s|^2)^{\mu}} d \iota (s)
\\ & = C_{\mu,\nu,k}^2 \sum_{l=0}^{k} \sum_{i=l}^k \binom{k}{i}^2 \binom{\nu}{k-i}^{-1} \binom{i}{l} (-1)^{i-l} \cdot
\\ & \int_{\C} \frac{(|s^2| + 1)^l f(s)}{(1 + |s|^2)^{\mu}} d \iota (s)
\\ & = C_{\mu,\nu,k}^2 \sum_{l = 0}^k (\sum_{i=l}^k \binom{k}{i}^2 \binom{\nu}{k-i}^{-1} \binom{i}{l} (-1)^{i-l}) B_{\mu - l}(f)(0).
\end{flalign*}
Now we study

$$\sum_{i=l}^k \binom{k}{i}^2 \binom{\nu}{k-i}^{-1} \binom{i}{l} (-1)^{i-l}.$$
We note that

\begin{flalign*}
& \sum_{i=l}^k \binom{k}{i}^2 \binom{\nu}{k-i}^{-1} \binom{i}{l} (-1)^{i-l}
\\ & = \sum_{i=l}^k (\frac{k!}{i!(k-i)})^2 \frac{(\nu - k + i)!(k-i)!}{\nu!} \frac{i!}{l!(i-l)!} (-1)^{i-l}
\\ & = \frac{k!}{\nu! l!} \sum_{i=0}^{k-l} \frac{k! (\nu - k + i + l)!}{(i+l)!(k-i-l)! i!} (-1)^{i}
\\ & = \frac{k!}{\nu! l!} \sum_{i=0}^{k-l} \frac{ (-k)_{l + i} (\nu - k + l)! (\nu - k + l + 1)_i }{l! (l+1)_i i!} (-1)^{l + i} (-1)^{i}
\\ & = (-1)^l \frac{k! (-k)_l (\nu - k + l)!}{\nu! (l!)^2} \sum_{i=0}^{k-l} \frac{(-k + l)_i (\nu - k + l + 1)_i}{i! (l+1)_i}
\\ & = \frac{ (k!)^2 (\nu - k + l )!}{\nu! (l!)^2 (k-l)!} {}_2F_1(-k + l, \nu - k + l + 1, l+1,1)
\\ & = \frac{(\nu - k + l )! (k!)^2}{\nu! (l!)^2 (k-l)!} \frac{(k- \nu)_{k-l}}{(l+1)_{k-l}}
\\ & = \frac{(\nu - k + l )! k! (k - \nu)_{k-l}}{\nu! l! (k-l)!}
\\ & = (-1)^{k-l} \binom{\nu - k}{k-l} \frac{(\nu - k + l)! k !}{\nu! l!}.
\end{flalign*}
Here we have used the Gauss formula for hypergeometric functions, Lemma \ref{Gausshgf}. Hence we see that $E_{\mu,k}^{\nu}(f)(0)$ is

\begin{flalign*}
& E_{\mu,k}^{\nu}(f)(0) = C_{\mu,\nu,k}^2 \sum_{l = 0}^k (-1)^{k-l} \binom{\nu - k}{k-l} \frac{(\nu - k + l)! k !}{\nu! l!} B_{\mu - l}(f)(0).
\end{flalign*}
The result follows.
\end{proof}

The constant $C_{\mu,\nu,k}^2 = \frac{ (-\nu)_k (-\mu)_k}{k! (\mu + \nu - 2k + 2)_k}$ has the property
$$ \lim_{\nu \rightarrow \infty} C_{\mu, \nu, k}^2 = \binom{\mu}{k}.$$
Thus

\begin{flalign*}
\lim_{\nu \rightarrow \infty} E_{\mu,k}^{\nu}(f) & = \binom{\mu}{k} \sum_{l = 0}^k (-1)^{k-l} \frac{k !}{(k-l)! l!} B_{\mu - l}(f)
\\ & = \binom{\mu}{k} \sum_{l=0}^k (-1)^{k-l} \binom{k}{l} B_{\mu - l}(f).
\end{flalign*}
We make this a definition.

\begin{definition}
\label{Emukdef}
$E_{\mu,k}(f) := \binom{\mu}{k} \sum_{l=0}^k (-1)^{k-l} \binom{k}{l} B_{\mu - l}(f).$
\end{definition}
Note the convergence $\lim_{\nu \rightarrow \infty} E_{\mu,k}^{\nu}(f) = E_{\mu,k}$ is uniform. We are interested in the case when $R_{\mu}^*(f) \geq 0$, for which we prove some bounds.

\begin{lemma}
\label{Ebounded}
For $f \in L^2(B_1)$, $z \in \C$ and $R_{\mu}^*(f) \geq 0$
$$ 0 \leq E_{\mu,k}(f)(z) \leq \Tr(R_{\mu}^*(f)).$$
\end{lemma}

\begin{proof}
We have for $X \in \Hc_{\mu}$
$$ \int_{\C} |X(s)|^2 f(s) d \iota_{\mu}(s) = (\mu + 1) \langle R_{\mu}^*(f)(X), X \rangle \geq 0.$$
Thus we find that
$$ E_{\mu,k}^{\nu}(f)(0) = \frac{C_{\mu,\nu,k}^2}{(\mu + 1)} \int_{\C^2} |(s - w)^k|^2 f(s) d \iota_{\nu}(w) d \iota_{\mu}(s) \geq 0,$$
and that for all $0 \leq j \leq \mu$
$$ \int_{\C} |s^j|^2 f(s) d \iota_{\mu}(s) = (\mu + 1) \langle R_{\mu}^*(f)(s^j), s^j \rangle \geq 0.$$
We also see that
\begin{flalign*}
E_{\mu,k}^{\nu}(f)(0) & = \frac{C_{\mu,\nu,k}^2}{(\mu + 1)} \int_{\C^2} |(s - w)^k|^2 f(s) d \iota_{\nu}(w) d \iota_{\mu}(s)
\\ & = \frac{C_{\mu,\nu,k}^2}{(\mu + 1)} \sum_{j=0}^k \binom{k}{j}^2 \binom{\nu}{k-j}^{-1} \int_{\C} |s^j|^2 f(s) d \iota_{\mu}(s).
\end{flalign*}
Taking the limit we see
\begin{flalign*}
E_{\mu,k}(f)(0) & = \lim_{\nu \rightarrow \infty} E_{\mu,k}^{\nu}(f)(0)
\\ & = \lim_{\nu \rightarrow \infty} \frac{C_{\mu,\nu,k}^2}{(\mu + 1)} \sum_{j=0}^k \binom{k}{j}^2 \binom{\nu}{k-j}^{-1} \int_{\C} |s^j|^2 f(s) d \iota_{\mu}(s)
\\ & = \binom{\mu}{k} \frac{1}{\mu + 1} \int_{\C} |s^k|^2 f(s) d \iota_{\mu}(s)
\\ & \leq \sum_{k=0}^{\mu} \binom{\mu}{k} \frac{1}{\mu + 1} \int_{\C} |s^k|^2 f(s) d \iota_{\mu}(s) = \frac{1}{\mu + 1} \int_{\C} (1 + |s|^2)^{\mu} f(s) d \iota_{\mu}(s)
\\ & = \int_{\C} f(s) d \iota(s) = \Tr(R_{\mu}^*(f)).
\end{flalign*}
We conclude that $0 \leq E_{\mu,k}(f)(0) \leq \Tr(R_{\mu}^*(f))$. But as the action of $SU(2)$ is transitive and $R_{\mu}^*(g \cdot f) = g R_{\mu}^*(f) g^{-1}$, which preserves the trace and positivity, we get that for any $z \in \C$
$$0 \leq E_{\mu,k}(f)(z) \leq \Tr(R_{\mu}^*(f)).$$
\end{proof}

\section{Functional calculus of quantum channels}
\label{funccalc}

We go on to describe the functional calculus of
$$\mathcal{T}_{\mu,k}^{\nu}(R_{\mu}^*(f)) = R_{\mu + \nu - 2k}^{-1} E_{\mu,k}^{\nu} (f).$$
for any $f \in C(\mathbb{CP}^1)$, where $R_{\mu + \nu - 2k}^{-1}$ is defined on $\mathrm{Im}(R_{\mu + \nu - 2k})$. We recall from Remark \ref{Rnuinjective} that $R_{\mu + \nu - 2k}^*$ restricted to this space is bijective, and we see

\begin{flalign*}
R_{\mu + \nu - 2k}^{-1} E_{\mu,k}^{\nu} & = R_{\mu + \nu - 2k}^* (R_{\mu + \nu - 2k}^*)^{-1} R_{\mu + \nu - 2k}^{-1} E_{\mu,k}^{\nu}
\\ & = R_{\mu + \nu - 2k}^* B_{\mu + \nu - 2k}^{-1} E_{\mu,k}^{\nu},
\end{flalign*}
where $B_{\mu + \nu - 2k}^{-1}$ is defined on $\mathrm{Im}(R_{\mu + \nu - 2k})$. Hence we see
\begin{equation}
\label{quantumchannelsberezininv}
\mathcal{T}_{\mu,k}^{\nu}(R_{\mu}^*(f)) = R_{\mu + \nu - 2k}^* B_{\mu + \nu - 2k}^{-1} E_{\mu,k}^{\nu}.
\end{equation}
We now want to prove that $((\nu + 1)B_{\nu})^{-1} f$ converges to $f$ as $\nu$ goes to infinity, where $f \in \mathrm{Im}(R_{\mu + \nu - 2k})$. We begin by proving $(\nu + 1) B_{\nu}$ is going to the identity.

\begin{proposition}
\label{berezintoid}
The sequence of operators $\{ (\nu + 1) B_{\nu}\}_{\nu}$, where $B_{\nu}$ is the Berezin transform, convergences to the indentity in the strong operator topology as an operator on $C(\mathbb{CP}^1)$.
\end{proposition}

\begin{proof}
Let $f \in C(\mathbb{CP}^1)$. First we look at $B_{\nu}(f)(0)$. We recall from Remark \ref{repsu2facts} that $\mathbb{CP}^1$ has a metric for which the action of $SU(2)$ is isometric. Let $\epsilon$ be given, now as $\mathbb{CP}^1$ is compact we can choose a $\delta > 0$ such that for $x,y \in \mathbb{CP}^1$ with $d(x,y) < \delta$ we have $|f(x) - f(y)| < \epsilon$. Let $U$ be the $\delta$-ball around $0$. Then we see

\begin{flalign*}
& | (\nu + 1)  B_{\nu}(f)(0) - f(0)|
\\ & = |(\nu + 1) \int_{\C} \frac{1}{(1 + |s|^2)^{\nu}} f(s) d \iota(s) - (\nu + 1) \int_{\C} \frac{1}{(1 + |s|^2)^{\nu}} f(0) d \iota(s)|
\\ & = | (\nu + 1) \int_{\C} \frac{1}{(1 + |s|^2)^{\nu}} (f(s) - f(0)) d \iota(s)|
\\ & \leq (\nu + 1) \int_{U} \frac{|f(s) - f(0)|}{(1 + |s|^2)^{\nu}} d \iota(s) + 2 \nm{f}_{\infty} \int_{\C \backslash U} \frac{ \nu + 1}{(1 + |s|^2)^{\nu}} d\iota(s).
\end{flalign*}

We also note that there is an $r > 0$ such that for all for $s \in \C \backslash U$ we have $s \geq r > 0$. We conclude that there is some $\nu_0 \in \mathbb{N}$ such that if $\nu \geq \nu_0$, then $| \frac{\nu + 1}{(1 + |s|^2)^{\nu}} | < \frac{\epsilon}{2 \nm{f}_{\infty}}$ for all $s \in \C \backslash U$. We then conclude that for $\nu \geq \nu_0$

\begin{flalign*}
& |(\nu + 1) B_{\nu}(f)(0) - f(0)|
\\ & \leq \epsilon (\nu + 1)  \int_{\C} \frac{1}{(1 + |s|^2)^{\nu}} d \iota(s) + \epsilon \int_{\C} d \iota(s) < 2 \epsilon.
\end{flalign*}
Then we see

\begin{flalign*}
& |(\nu + 1) B_{\nu}(f)(g \cdot 0) - f(g \cdot 0)| = | (\nu + 1) B_{\nu}(g^{-1} \cdot f)(0) - (g^{-1} \cdot f)(0)|.
\end{flalign*}
Now we see for $g \in SU(2)$ and $z \in U$ that $d(g \cdot 0, g \cdot z) < \delta$, so

$$|g^{-1} \cdot f(0) - g^{-1} \cdot f(z)| = |f(g \cdot 0) - f(g \cdot z)| < \epsilon.$$
It follows that for the same $\nu_0$, if $\nu \geq \nu_0$:
$$ |(\nu + 1)B_{\nu}(f)(g) - f(g)| = | (\nu + 1)B_{\nu}(g^{-1} \cdot f)(0) - (g^{-1} \cdot f) (0) | < 2 \epsilon.$$
We conclude that
$$ \lim_{\nu \rightarrow \infty} \nm{B_{\nu}(f) - f} = 0.$$
\end{proof}

Now we look at $( (\nu + 1) B_{\nu})^{-1}(f)$, where for $f \in \mathrm{Im}(R_{\mu})$ and $\nu \geq \mu$. Note that $\mathrm{Im}(R_{\mu})$ is a finite dimensional vector space. 

\begin{proposition}
\label{berezininverse}
For $f \in \mathrm{Im}(R_{\mu})$ we have that $((\nu + 1)B_{\nu})^{-1}(f)$ converges uniformly to $f$ as $\nu$ goes to infinity. In particular, the convergence of $( (\mu + \nu - 2k + 1) B_{\nu + \mu - 2k})^{-1} E_{\mu,k}^{\nu}(f)$ to $E_{\mu,k}(f)$ is uniform in $C(\mathbb{CP}^1)$.
\end{proposition}

\begin{remark}
We note that $C^{\infty}(\mathbb{CP}^1) \subseteq L^2(\mathbb{CP}^1) \cong \bigoplus_{k = 0}^{\infty} \Hc_{2k}$ and $B(\Hc_{\mu}) \cong \bigoplus_{k=0}^{\mu} \Hc_{2k}$, where $B_{\mu}$ is injective, so that $C^{\infty}(\mathbb{CP}^1) \supseteq \mathrm{Im}(R_{\mu}) \cong \bigoplus_{k=0}^{\mu} \Hc_{2k}$. Hence if $\nu \geq \mu$ we have
$$ \mathrm{Im}(R_{\mu}) \subseteq \mathrm{Im}(R_{\nu}).$$
We conclude that $B_{\nu}$ maps $\mathrm{Im}(R_{\mu})$ onto itself and $B_{\nu}^{-1}(f)$ is well-defined on $\mathrm{Im}(R_{\mu})$ when $\nu \geq \mu$.
\end{remark}

\begin{proof}
By Proposition \ref{berezintoid}
$$ \lim_{\nu \rightarrow \infty} (\nu+1) B_{\nu}|_{\mathrm{Im}(R_{\mu})} = I_{\nu}$$
in the strong operator topology. As $\mathrm{Im}(R_{\mu})$ is of fixed finite dimension, this convergence is also in the norm (the norm being the inherited $\nm{-}_{\infty}$ norm). Hence we have some $\nu_0$ such that when $\nu \geq \nu_0$:
$$ \nm{((\nu + 1)B_{\nu}|_{\mathrm{Im}(R_{\mu})})^{-1} - I} < \epsilon. $$
Thus we have that for $\nu \geq \nu_0$:
$$ ((\nu + 1) B_{\nu}|_{\mathrm{Im}(R_{\mu})})^{-1} = \sum_{n = 0}^{\infty} (I - (\nu+1) B_{\nu}|_{\mathrm{Im}(R_{\mu})})^n,$$
and we get that
$$\nm{ ( (\nu+1) B_{\nu}|_{\mathrm{Im}(R_{\mu})})^{-1}} \leq \frac{1}{1 + \nm{I - (\nu + 1) B_{\nu}||_{\mathrm{Im}(R_{\mu})}}},$$
which is bounded. Hence we have that
\begin{flalign*}
& \nm{f - ( (\nu+1) B_{\nu})^{-1}(f)}_{\infty} = \nm{( (\nu+1) B_{\nu}|_{\mathrm{Im}(R_{\mu})})^{-1}( (\nu+1) B_{\nu}(f) - f)}_{\infty}
\\ & \leq \nm{((\nu+1) B_{\nu}|_{\mathrm{Im}(R_{\mu})})^{-1}} \cdot \nm{ (\nu+1) B_{\nu}(f) - f }_{\infty}.
\end{flalign*}
This proves that $( (\nu + 1) B_{\nu})^{-1}(f)$ converges uniformly to $f$.

For the second part, we know that $E_{\mu,k}^{\nu}(f)$ is going to $E_{\mu,k}(f)$ uniformly, so using the fact that $E_{\mu,k}^{\mu+\nu}(f), E_{\mu,k}(f) \in \sum_{l=0}^k \mathrm{Im}(B_{\mu - l}) = \mathrm{Im}(R_{\mu})$ we get

\begin{flalign*}
& \nm{( (\nu + \mu - 2k) B_{\mu + \nu - 2k})^{-1} E_{\mu,k}^{\nu}(f) - E_{\mu,k}(f)}_{\infty}
\\ & \leq \nm{(( \nu + \mu - 2k) B_{\nu + \mu - 2k})^{-1}E_{\mu,k}^{\nu}(f) - E_{\mu,k}^{\nu}(f)}_{\infty} + \nm{E_{\mu,k}^{\nu}(f) - E_{\mu,k}(f)}_{\infty}
\\ & \leq \nm{ ((\nu + \mu - 2k) B_{\nu + \mu - 2k})^{-1} - I} \cdot \nm{E_{\mu,k}^{\nu}(f)}_{\infty} + \nm{E_{\mu,k}^{\nu}(f) - E_{\mu,k}(f)}_{\infty}.
\end{flalign*}
The result follows.
\end{proof}

We want to prove convergence of $R_{\nu}( R_{\nu}^*(f)^n)$ now.

\begin{lemma}
\label{boundedkern}
For each $n$ in $\mathbb{N}$ the integral
$$I_n(\nu) := (\nu + 1)^n \int_{\C^n} \left\lvert \frac{(1 + s_1 \li{s_2}) \dots (1 + s_{n-1} \li{s_n})}{(1 + |s_1|^2) \dots (1 + |s_n|^2)} \right\lvert^{\nu} d \iota(s_1) \dots d \iota(s_n)$$
is bounded in $\nu \in \mathbb{N}$. More precisely, $\forall_{\nu \in \mathbb{N}} \ I_n(\nu) \leq 2^{2n}$.
\end{lemma}

\begin{proof}
It is clear when $n=1$, then the integrand is $(\frac{1}{1 + |s_1|^2})^{\nu}$. Assume now $n \geq 2$. We see for $\nu$ even, i.e. $\nu = 2 \kappa$
\begin{flalign*}
I_n(\nu) & = \int_{\C^n} |(1 + s_1 \li{s_2}) \dots (1 + s_{n-1} \li{s_n})|^{2 \kappa} d \iota_{\nu}(s_1) \dots d \iota_{\nu}(s_n)
\\ & = \nm{(1 + s_1 s_2)^{\kappa} \dots (1 + s_{n-1} s_n)^{\kappa}}_{\Hc_{\nu} \otimes \dots \otimes \Hc_{\nu}}^2
\\ & = \sum_{i_1, \dots, i_{n-1}} \binom{\kappa}{i_1}^2 \dots \binom{\kappa}{i_{n-1}}^2 \cdot
\\ & \binom{2 \kappa}{i_1}^{-1} \binom{2 \kappa}{i_1 + i_2}^{-1} \dots \binom{2 \kappa}{i_{n-2} + i_{n-1}}^{-1} \binom{2 \kappa}{i_{n-1}}^{-1}.
\end{flalign*}
Now we note that
\begin{flalign*}
& \frac{2 \kappa + 1}{\kappa + 1} \binom{\kappa}{j}^{-1} = \frac{2 \kappa + 1}{\kappa + 1} \nm{x^j}_{\kappa}^2 = (2 \kappa + 1) \int_{\C} \frac{|x|^{2j}}{(1 + |x|^{2})^{\kappa}} d \iota(x)
\\ & = \int_{\C} |x|^{2j} (1 + |x|^2)^{\kappa} d \mu_{2 \kappa}(x) = \sum_{i= 0}^{\kappa} \binom{\kappa}{i} \nm{x^{i+j}}_{2 \kappa}^2 = \sum_{i= 0}^{\kappa} \binom{\kappa}{i} \binom{2 \kappa}{i + j}^{-1},
\end{flalign*}
and thus
\begin{equation}
\label{fundineq}
\sum_{i= 0}^{\kappa} \binom{\kappa}{i} \binom{2 \kappa}{i + j}^{-1} \leq 2 \binom{\kappa}{j}^{-1}.
\end{equation}

We continue our calculations on $I_n(\nu)$. We have
\begin{flalign*}
& \sum_{i_1, \dots, i_{n-1}} \binom{\kappa}{i_1}^2 \dots \binom{\kappa}{i_{n-1}}^2 \cdot
\\ & \binom{2 \kappa}{i_1}^{-1} \binom{2 \kappa}{i_1 + i_2}^{-1} \dots \binom{2 \kappa}{i_{n-2} + i_{n-1}}^{-1} \binom{2 \kappa}{i_{n-1}}^{-1}
\\ & = \sum_{i_2, \dots, i_{n-1}} (\sum_{i_1 = 0}^{\kappa} \binom{\kappa}{i_1}^2 \binom{2 \kappa}{i_1}^{-1} \binom{2 \kappa}{i_1 + i_2}^{-1}) \binom{\kappa}{i_2}^2 \dots \binom{\kappa}{i_{n-1}}^2 \cdot
\\ & \binom{2 \kappa}{i_1 + i_2}^{-1} \dots \binom{2 \kappa}{i_{n-2} + i_{n-1}}^{-1} \binom{2 \kappa}{i_{n-1}}^{-1}.
\end{flalign*}
Now we see that
\begin{flalign*}
& \sum_{i_1 = 0}^{\kappa} \binom{\kappa}{i_1}^2 \binom{2 \kappa}{i_1}^{-1} \binom{2 \kappa}{i_1 + i_2}^{-1} \leq \sum_{i_1 = 0}^{\kappa} \binom{\kappa}{i_1} \binom{2 \kappa}{i_1}^{-1} \sum_{i_1=0}^{\kappa} \binom{\kappa}{i_1} \binom{2 \kappa}{i_1 + i_2}^{-1}
\\ & \leq 4 \binom{\kappa}{i_2}^{-1} \binom{\kappa}{0}^{-1} = 4 \binom{\kappa}{i_2}^{-1}.
\end{flalign*}
Thus we see applying Inequality \ref{fundineq} iteratively
\begin{flalign*}
& \sum_{i_2, \dots, i_{n-1}} (\sum_{i_1 = 0}^{\kappa} \binom{\kappa}{i_1}^2 \binom{2 \kappa}{i_1}^{-1} \binom{2 \kappa}{i_1 + i_2}^{-1}) \binom{\kappa}{i_2}^2 \dots \binom{\kappa}{i_{n-1}}^2 \cdot
\\ & \binom{2 \kappa}{i_1 + i_2}^{-1} \dots \binom{2 \kappa}{i_{n-2} + i_{n-1}}^{-1} \binom{2 \kappa}{i_{n-1}}^{-1}.
\\ & \leq 4 \sum_{i_2, \dots, i_{n-1}} \binom{\kappa}{i_2} \binom{\kappa}{i_3}^2 \dots \binom{\kappa}{i_{n-1}}^2 \cdot
\\ & \binom{2 \kappa}{i_2 + i_3}^{-1} \dots \binom{2 \kappa}{i_{n-2} + i_{n-1}}^{-1} \binom{2 \kappa}{i_{n-1}}^{-1}
\\ & = 4 \sum_{i_3, \dots, i_n} (\sum_{i_2=0}^{\kappa} \binom{\kappa}{i_2} \binom{2 \kappa }{i_2 + i_3}^{-1}) \binom{\kappa}{i_3}^2 \dots \binom{\kappa}{i_{n-1}}^2 \cdot
\\ & \binom{2 \kappa}{i_3 + i_4}^{-1} \dots \binom{2 \kappa}{i_{n-2} + i_{n-1}}^{-1} \binom{2 \kappa}{i_{n-1}}^{-1}
\\ & \leq 8 \sum_{i_3, \dots, i_n} \binom{\kappa}{i_3} \binom{\kappa}{i_4}^2 \dots \binom{\kappa}{i_{n-1}}^2 \cdot
\\ & \binom{2 \kappa}{i_3 + i_4}^{-1} \dots \binom{2 \kappa}{i_{n-2} + i_{n-1}}^{-1} \binom{2 \kappa}{i_{n-1}}^{-1}
\\ & \leq \dots \leq 2^{n-1} \sum_{i_{n-1}=0}^{\kappa} \binom{\kappa}{i_{n-1}} \binom{2 \kappa}{i_{n-1}}^{-1} \leq 2^{n} \binom{\kappa}{0}^{-1} = 2^n.
\end{flalign*}
Consider now the case that $\nu$ is odd. Then
\begin{flalign*}
I_n(\nu) & \leq \nu^n (\frac{\nu + 1}{\nu})^n \int_{\C^n} \left\lvert \frac{(1 + s_1 \li{s_2}) \dots (1 + s_{n-1} \li{s_n})}{(1 + |s_1|^2) \dots (1 + |s_n|^2)} \right\rvert^{\nu-1} d \iota(s_1) \dots d \iota(s_n)
\\ & \leq 2^n (\frac{\nu + 1}{\nu})^n \leq 2^{2n}.
\end{flalign*}
We conclude that indeed the desired integral is bounded.
\end{proof}

\begin{proposition}
\label{berezinxncpt} For $f \in C(\mathbb{CP}^1)$ we have the convergence 
$$\lim_{\nu \rightarrow \infty} \nm{(\nu + 1)^n R_{\nu}( R_{\nu}^*(f)^n) - f^n}_{\infty} = 0.$$
\end{proposition}

\begin{proof}
Recall by $SU(2)$-invariance
\begin{flalign*}
& 
R_{\nu}( R_{\nu}^*(f)^n) (g \cdot 0) = R_{\nu}( g^{-1} (R_{\nu}^*(f)^n))(0)
\\ & =  R_{\nu}( (g^{-1}R_{\nu}^*(f)g)^n)(0) = R_{\nu}( R_{\nu}^*(g^{-1}f)^n)(0).
\end{flalign*}
Hence we can do the calculations by evaluating in $0$ as in the proof of Proposition \ref{berezintoid}. We see that the kernel is
\begin{flalign*}
& R_{\nu}^*(f)^n(x,y) = \int_{\C^{2n-1}} \frac{f(s_1) \dots f(s_n) (1 + x \li{s_1})^{\nu} (1 + s_1 \li{t_1})^{\nu} \dots (1 + s_n \li{y})^{\nu}}{(1 + |s_1|^2)^{\nu} \dots (1 + |s_n|^2)^{\nu}} \cdot 
\\ & (1 + t_1 \li{s_2})^{\nu} \dots (1 + t_{n-1} \li{s_{n}})^{\nu} d \iota(s_1) \dots d \iota(s_n) d \iota_{\nu}(t_1) \dots d \iota_{\nu}(t_{n-1})
\\ & = \int_{\C^n} \frac{f(s_1) \dots f(s_n) (1 + x \li{s_1})^{\nu} (1 + s_1 \li{s_2})^{\nu} \dots (1 + s_n \li{y})^{\nu}}{(1 + |s_1|^2)^{\nu} \dots (1 + |s_n|^2)^{\nu}} d \iota(s_1) \dots d \iota(s_n).
\end{flalign*}
Then
\begin{flalign*}
& (\nu + 1)^n R_{\nu} (R_{\nu}^*(f)^n)(x)
\\ & = (\nu + 1)^n \frac{R_{\nu}^*(f)^n(x,x)}{(1 + |x|^2)^{\nu}}
\\ & = \frac{(\nu + 1)^n}{1 + |x|^2} \int_{\C^n} \frac{f(s_1) \dots f(s_n) (1 + x \li{s_1})^{\nu} \dots (1 + s_n \li{x})^{\nu}}{(1 + |s_1|^2)^{\nu} \dots (1 + |s_n|^2)^{\nu}} d \iota (s_1) \dots d \iota(s_n)
\\ & = (\nu + 1)^{n} \int_{\C^n} \frac{f(s_1) \dots f(s_n) (1 + x \li{s_1})^{\nu} \dots (1 + s_n \li{x})^{\nu}}{(1 + |x|^2)^{\nu} (1 + |s_1|^2)^{\nu} \dots (1 + |s_n|^2)^{\nu} } d \iota(s_1) \dots d \iota(s_n)
\end{flalign*}
Evaluating in $0$
\begin{flalign*}
& (\nu + 1)^n R_{\nu} (R_{\nu}^*(f)^n)(0)
\\ & = (\nu + 1)^n \int_{\C^n} \frac{f(s_1) \dots f(s_n) (1 + s_1 \li{s_2})^{\nu} \dots (1 + s_{n-1} \li{s_n})^{\nu}}{(1 + |s_1|^2)^{\nu} \dots (1 + |s_n|^2)^{\nu} } d \iota(s_1) \dots d \iota(s_n).
\end{flalign*}
We also note that expanding $(1 + s_1 \li{s_2})^{\nu} \dots (1 + s_{n-1} \li{s_n})^{\nu}$, only the holomorphic parts of $s_1$ appear and by orthogonality they vanish. Then we only have the holomorphic factors of $s_2$, and iterating only the constant term is integrated over. We get
\begin{flalign*}
& (\nu + 1)^{n} \int_{\C^n} \frac{(1 + s_1 \li{s_2})^{\nu} \dots (1 + s_{n-1} \li{s_n})^{\nu}}{(1 + |s_1|^2)^{\nu} \dots (1 + |s_n|^2)^{\nu} } d \iota(s_1) \dots d \iota(s_n)
\\ & = (\nu + 1)^{n} \int_{\C^n} \frac{1}{(1 + |s_1|^2)^{\nu} \dots (1 + |s_n|^2)^{\nu} } d \iota(s_1) \dots d \iota(s_n) = 1
\end{flalign*}

Now we take the sum metric on $\C^n$, inherited by the $SU(2)$-invariant metric on $\C$, invariant under the diagonal action of $SU(2)$, and as we are on a compact space for every $\epsilon > 0$ there is a $\delta > 0$ such that $d((x_1, \dots, x_n),(y_1, \dots, y_n)) < \delta$ implies $|f(x_1) \dots f(x_n) - f(y_1) \dots f(y_n)| < \epsilon$. We fix $\epsilon$ giving us a $\delta$, and we let $U$ be the open $\delta$-ball around $0$,

\begin{flalign*}
& | (\nu + 1)^n  R_{\nu}( R_{\nu}^*(f)^n)(0) - f(0)^n|
\\ & = |(\nu + 1)^n \int_{\C^n} \frac{f(s_1) \dots f(s_n) (1 + s_1 \li{s_2})^{\nu} \dots (1 + s_{n-1} \li{s_n})^{\nu}}{(1 + |s_1|^2)^{\nu} \dots (1 + |s_n|^2)^{\nu} } d \iota(s_1) \dots d \iota(s_n)
\\ & - (\nu + 1)^n \int_{\C^n} \frac{f(0)^n (1 + s_1 \li{s_2})^{\nu} \dots (1 + s_{n-1} \li{s_n})^{\nu}}{(1 + |s_1|^2)^{\nu} \dots (1 + |s_n|^2)^{\nu} } d \iota(s_1) \dots d \iota(s_n)|
\\ & = | (\nu + 1)^n \int_{\C^n} \frac{(f(s_1) \dots f(s_n) - f(0)^n) (1 + s_1 \li{s_2})^{\nu} \dots (1 + s_{n-1} \li{s_n})^{\nu}}{(1 + |s_1|^2)^{\nu} \dots (1 + |s_n|^2)^{\nu}} \cdot
\\ & d \iota(s_1) \dots d \iota(s_n)|
\\ & \leq (\nu + 1)^n \int_{U} |f(s_1) \dots f(s_n) - f(0)^n| \cdot \left\lvert \frac{(1 + s_1 \li{s_2}) \dots (1 + s_{n-1} \li{s_n})}{(1 + |s_1|^2) \dots (1 + |s_n|^2)} \right\rvert^{\nu} \cdot
\\ & d \iota(s_1) \dots d \iota(s_n)
\\ & + 2 \nm{f^n}_{\infty} \int_{\C^n \backslash U} (\nu + 1)^{n} \left\lvert \frac{(1 + s_1 \li{s_2}) \dots (1 + s_{n-1} \li{s_n}) }{(1 + |s_1|^2) \dots (1 + |s_n|^2)} \right\rvert^{\nu} d\iota(s_1) \dots d \iota(s_n).
\end{flalign*}
It follows by the Cauchy-Schwarz inequality
$$ |1 + s_1 \li{s_2}| \leq (1 + |s_1|^2)^{\frac{1}{2}} (1 + |s_2|^2)^{\frac{1}{2}}$$
that
$$ \left\lvert \frac{(1 + s_1 \li{s_2}) \dots (1 + s_n \li{s_1}) }{(1 + |s_1|^2) \dots (1 + |s_n|^2)} \right\rvert^{\nu} \leq \frac{1}{(1 + |s_1|^2)^{\frac{\nu}{2}} (1 + |s_n|^2)^{\frac{\nu}{2}}} \leq 1.$$
We have equality in the first inequality if and only if all $s_i$ are equal; in the second inequality if and only if all $s_i$ are equal to $0$. In $\C^n \backslash U$ we have that there exists $r > 1$ such that
$$ \frac{1}{(1 + |s_1|^2)^{\frac{\nu}{2}} (1 + |s_n|^2)^{\frac{\nu}{2}}} \leq r < 1.$$
It follows that
$$(s_1, \dots, s_n) \mapsto (\nu + 1)^{n} \left\lvert \frac{(1 + s_1 \li{s_2}) \dots (1 + s_{n-1} \li{s_n}) }{(1 + |s_1|^2) \dots (1 + |s_n|^2)} \right\rvert^{\nu} < (\nu + 1) r^n$$
is bounded on $\C^n \backslash U$ and is going to $0$. This gives us by Lebesgue dominated convergence that $\int_{\C^n \backslash U} (\nu + 1)^{n} |\frac{(1 + s_1 \li{s_2}) \dots (1 + s_{n-1} \li{s_n}) }{(1 + |s_1|^2) \dots (1 + |s_n|^2)}|^{\nu} d\iota(s_1) \dots d \iota(s_n)$ is going to $0$ as $\nu$ is going to infinity. Now we study
\begin{flalign*}
& (\nu + 1)^n \times
\\ & \int_{U} | (f(s_1) \dots f(s_n) - f(0)^n) \frac{(1 + s_1 \li{s_2}) \dots (1 + s_{n-1} \li{s_n})}{(1 + |s_1|^2) \dots (1 + |s_n|^2)} |^{\nu} d \iota(s_1) \dots d \iota(s_n).
\end{flalign*}
We see that for $(s_1, \dots, s_n) \in U$ we have
$$|f(s_1) \dots f(s_n) - f(0)^n| < \epsilon.$$
As a consequence of Lemma \ref{boundedkern} the integral
$$(\nu + 1)^n \int_{U} \left\lvert \frac{(1 + s_1 \li{s_2}) \dots (1 + s_{n-1} \li{s_n})}{(1 + |s_1|^2) \dots (1 + |s_n|^2)} \right\rvert^{\nu} d \iota(s_1) \dots d \iota(s_n)$$
is bounded by some $R$. This bound holds for each neighbourhood $U$. Thus there is some $\nu_0$ such that for $\nu \geq \nu_0$
$$ | (\nu + 1)^n  R_{\nu}( R_{\nu}^*(f)^n)(0) - f(0)^n| < \epsilon + R \epsilon.$$
The result follows like in Proposition \ref{berezintoid}.
\end{proof}

We combine everything to get the following Theorem.

\begin{theorem}
\label{main_theorem}
For $f \in C(\mathbb{CP}^1)$ there is trace convergence

$$ \lim_{\nu \rightarrow \infty} \frac{1}{\nu} \Tr( \mathcal{T}_{\mu,k}^{\nu}(R^*_{\mu}(f))^n) = \int_{\C} E_{\mu,k}(f)^n d \iota(z).$$
\end{theorem}

\begin{proof}
We define
$$g_{\nu} := (\nu B_{\mu + \nu - 2k})^{-1} E_{\mu,k}^{\nu}(f)$$
and
$$g := E_{\mu,k}(f).$$
By Proposition \ref{berezininverse} we have that $g_{\nu}$ converges to $g$. Note that here we have exchanged $\nu + \mu - 2k + 1$ for $\nu$ as per Remark \ref{munuktonu}. We see
\begin{flalign*}
& \nm{\nu^n R_{\mu + \nu - 2k}( R_{\mu + \nu - 2k}^*(g_{\nu})^n) - g^n}_{\infty}
\\ & \leq \nm{\nu^n R_{\mu + \nu - 2k}( R_{\mu + \nu - 2k}^*(g_{\nu})^n) - \nu^n R_{\mu + \nu - 2k}( R_{\mu + \nu - 2k}^*(g)^n)}_{\infty} 
\\ & + \nm{\nu^n R_{\mu + \nu - 2k}( R_{\mu + \nu - 2k}^*(g)^n) - g^n}_{\infty}
\\ & \leq \nm{\nu^n R_{\mu + \nu - 2k}( R_{\mu + \nu - 2k}^*(g_{\nu})^n - R_{\mu + \nu - 2k}^*(g)^n)}_{\infty}
\\ & + \nm{\nu^n R_{\mu + \nu - 2k}( R_{\mu + \nu - 2k}^*(g)^n) - g^n}_{\infty}.
\end{flalign*}
By Proposition \ref{berezinxncpt} we have
$$ \lim_{\nu \rightarrow \infty} \nm{\nu^n R_{\mu + \nu - 2k}( R_{\mu + \nu - 2k}^*(g)^n) - g^n}_{\infty} = 0.$$
Now we see that
$$ |A(z,z)| = |\langle A, K_z \otimes \li{K_z} \rangle_{B(\Hc_{\nu})}| \leq \nm{A} \cdot \nm{K_z \otimes \li{K_z}} = \nm{A} (1 + |z|^2)^{\nu},$$
so that $\nm{R_{\nu}(A)}_{\infty} \leq \nm{A}_{B(\Hc_{\nu})}$. Furthermore, as $\nu R^*_{\mu + \nu - 2k}(f)$ is $c_{\nu} T_f$, $T_f$ being the Toeplitz operator and $c_{\nu} = \frac{\nu}{\nu + \mu - 2k +1}$, we see that
$$\nm{\nu R^*_{\mu + \nu - 2k}(f)}_{B(\Hc_{\nu})} \leq c_{\nu} \nm{f}_{\infty},$$
and $\lim_{\nu \rightarrow \infty} c_{\nu} = 1$. It follows that
\begin{flalign*}
& \lim_{\nu \rightarrow \infty} \nm{\nu^n R_{\mu + \nu - 2k}( R_{\mu + \nu - 2k}^*(g_{\nu})^n - R_{\mu + \nu - 2k}^*(g)^n)}_{\infty}
\\ & = \lim_{\nu \rightarrow \infty} \nm{R_{\mu + \nu - 2k}( T_{g_{\nu}}^n - T_g^n)}_{\infty} = 0.
\end{flalign*}
Thus we conclude that
$$ \lim_{\nu \rightarrow \infty} \nm{R_{\mu + \nu - 2k}( R_{\mu + \nu - 2k}^*(g_{\nu})^n) - g^n}_{\infty} = 0. $$
Now by Equation \ref{quantumchannelsberezininv}
\begin{flalign*}
& \frac{1}{\nu} \Tr(\mathcal{T}_{\mu,k}^{\nu}(R_{\mu}^*(f))^n) = \frac{1}{\nu} \Tr( \nu^n (R_{\mu + \nu - 2k}^* (\nu B_{\mu + \nu - 2k + 1})^{-1} E_{\mu,k}^{\nu}(f))^n)
\\ & = \nu^{n} c_{\nu}^{-1} \int_{\C}  R_{\mu + \nu - 2k} (R_{\mu + \nu - 2k}^*( (\nu B_{\mu + \nu - 2k})^{-1} E_{\mu,k}^{\nu}(f))^n)(z) d \iota(z)
\\ & = \nu^{n} c_{\nu}^{-1} \int_{\C}  R_{\mu + \nu - 2k} (R_{\mu + \nu - 2k}^*( g_{\nu})^n)(z) d \iota(z).
\end{flalign*}
As $d \iota$ is a bounded measure
\begin{flalign*}
& \lim_{\nu \rightarrow \infty} \nu^{n} c_{\nu}^{-1} \int_{\C}  R_{\mu + \nu - 2k} (R_{\mu + \nu - 2k}^*( g_{\nu})^n)(z) d \iota(z)
\\ & = \int_{\C} g^n d \iota(z) =  \int_{\C} E_{\mu,k}(f)^n d \iota(z).
\end{flalign*}
The result follows.
\end{proof}

Now we prove our main result.

\begin{theorem}
Let $\phi \in C([0,1])$. Then
$$ \lim_{\nu \rightarrow \infty} \frac{1}{\nu } \Tr(\phi(\mathcal{T}_{\mu,k}^{\nu}(R_{\mu}^*(f)))) = \int_{\C} \phi(E_{\mu,k}(f)) d \iota(z),$$
for $f \in C(\mathbb{CP}^1)$ such that $R_{\mu}^*(f)$ positive and $\int_{\C} f(z) d \iota(z) = 1$.
\end{theorem}

\begin{remark}
We observe that the condition $\int_{\C} f(z) d \iota(z) = 1$ is equivalent to the condition $\Tr(R_{\mu}^*(f)) = 1$. Note that the eigenvalues of $\mathcal{T}_{\mu,k}^{\nu}(R_{\mu}^*(f))$ are smaller than or equal to $\nm{R_{\mu}^*(f)}$ for all $\nu$. We get this as
$$ \nm{\mathcal{T}_{\mu,k}^{\nu}(R_{\mu}^*(f))} = \nm{P_k (R_{\mu}^*(f) \otimes ) P_k^*} \leq \nm{R_{\mu}^*(f)}.$$
Note also that
$$ \nm{R_{\mu}^*(f)} \leq \Tr(R_{\mu}^*(f)),$$
as $R_{\mu}^*(f)$ is positive and $\Tr(R_{\mu}^*(f)) = 1$. Thus the functional calculus $\phi(\mathcal{T}_{\mu,k}^{\nu}(R_{\mu}^*(f)))$ for $\phi \in C([0,1])$ is well defined. Note also that $\phi(E_{\mu,k}(f))$ is well defined by Lemma \ref{Ebounded}.
\end{remark}

\begin{proof}
We use Weierstrass' theorem on the density of polynomials in $C([0,1])$. Let $\phi \in C([0,1])$ and $\{p_n\}_n$ a sequence of polynomials such that
$$ \lim_{n \rightarrow \infty} \nm{\phi - p_n}_{\infty} = 0.$$
Observe
\begin{flalign*}
& |\frac{1}{\nu } \Tr(\phi(\mathcal{T}_{\mu,k}^{\nu}(R_{\mu}^*(f)))) - \int_{\C} \phi(E_{\mu,k}(f)) d \iota(z)|
\\ & \leq |\frac{1}{\nu } \Tr((\phi - p_n)(\mathcal{T}_{\mu,k}^{\nu}(R_{\mu}^*(f))))|
\\ & + |\frac{1}{\nu } \Tr(p_n(\mathcal{T}_{\mu,k}^{\nu}(R_{\mu}^*(f)))) - \int_{\C} p_n(E_{\mu,k}(f)) d \iota(z)|
\\ & + |\int_{\C} (p_n - \phi)(E_{\mu,k}(f)) d \iota(z)|.
\end{flalign*}
By Theorem \ref{main_theorem} we know
$$ \lim_{\nu \rightarrow \infty} \frac{1}{\nu } \Tr(p_n(\mathcal{T}_{\mu,k}^{\nu}(R_{\mu}^*(f)))) = \int_{\C} p_n(E_{\mu,k}(f)) d \iota(z).$$
Furthermore
$$ \lim_{n \rightarrow \infty} |\int_{\C} (p_n - \phi)(E_{\mu,k}(f)) d \iota(z)| \leq \lim_{n \rightarrow \infty} \int_{\C} \nm{p_n - \phi}_{\infty} d \iota(z)= 0.$$
Now let $\{\lambda_i\}_{i=1}^{\mu + \nu - 2k}$ be the eigenvalues of $(\phi - p_n)(\mathcal{T}_{\mu,k}^{\nu}(R_{\mu}^*(f)))$. We see that
$$ \nm{(\phi - p_n)(\mathcal{T}_{\mu,k}^{\nu}(R_{\mu}^*(f)))} \leq \nm{\phi - p_n}_{\infty},$$
and thus for all $i$
$$|\lambda_i| \leq \nm{\phi - p_n}_{\infty}.$$
We conclude that
$$ |\frac{1}{\nu } \Tr((\phi - p_n)(\mathcal{T}_{\mu,k}^{\nu}(R_{\mu}^*(f))))| \leq \frac{\nu + \mu - 2k + 1}{\nu} \nm{\phi - p_n}_{\infty},$$
and thus that
$$ \lim_{n \rightarrow \infty} |\frac{1}{\nu } \Tr((\phi - p_n)(\mathcal{T}_{\mu,k}^{\nu}(R_{\mu}^*(f))))| = 0.$$
This proves that
$$ \lim_{\nu \rightarrow \infty} \frac{1}{\nu } \Tr(\phi(\mathcal{T}_{\mu,k}^{\nu}(R_{\mu}^*(f)))) = \int_{\C} \phi(E_{\mu,k}(f)) d \iota(z).$$
\end{proof}

We study $E_{\mu,k}$ to find its spectral decomposition under the decomposition $L^2(\mathbb{CP}^1) = \bigoplus_{k=0}^{\infty} \widetilde{\Hc_{2k}}$. It is a sum of Berezin transforms and thus $SU(2)$-invariant. Schur's lemma says it will act by a constant on each of the irreducible subspaces $\widetilde{\Hc_{2m}} \subseteq L^2(\mathbb{CP}^1)$. We can compute this constant precisely.

\begin{theorem}
The operator $E_{\mu,k}$ acts as the constant
$$\frac{(-1)^k \binom{\mu}{k} (\mu !)^2 {}_3 F_2 (-k, -m - \mu - 1, m - \mu ; - \mu, - \mu ; 1)}{(\mu - m)! (\mu + m + 1)!}$$
on the space $\widetilde{\Hc_{2 m}} \subseteq C^{\infty}(\mathbb{CP}^1)$, where $0 \leq \mu \leq m$. If $m > \mu$ it acts as $0$.
\end{theorem}

\begin{proof}
We recall the operator $E_{\mu,k}$ from Definition \ref{Emukdef}. From Equation (\ref{Bnu0}) and the definition of $E_{\mu,k}$ it is clear that it acts as $0$ when $m > \mu$. Now we study what happens for $0 \leq m \leq \mu$. By Equation (\ref{Bnuconstant}) $B_{\tau}$ acts as the constant $\frac{(\tau!)^2}{(\tau + m + 1)! \cdot (\tau - m)!}$ on $\Hc_{2m}$, and thus $E_{\mu,k}$ acts as the constant

$$ \binom{\mu}{k} \sum_{l = 0}^{\mu - m} (-1)^{k-l} \binom{k}{l} \frac{ ((\mu-l)!)^2}{(\mu - l + m + 1)! \cdot (\mu - l - m)!}.$$
This is also equal to

\begin{flalign*}
& \binom{\mu}{k} (-1)^k \sum_{l=0}^{\mu - m} (-1)^{l} \binom{k}{l} \frac{ ((\mu-l)!)^2}{(\mu - l + m + 1)! \cdot (\mu - l - m)!}
\\ & = \binom{\mu}{k} (-1)^k \frac{ (\mu!)^2}{(\mu + m + 1)! (\mu - m)!} \sum_{l=0}^{\mu - m} (-1)^{l} \frac{(\mu + m + 1)^l_- (\mu - m)^l_- (k)^l_-}{((\mu)^l_-)^2 l!}
\\ & = \binom{\mu}{k} (-1)^k \frac{ (\mu!)^2}{(\mu + m + 1)! (\mu - m)!} \sum_{l=0}^{\mu - m} \frac{(- \mu - m - 1)_l (- \mu + m)_l (-k)_l}{((-\mu)_l)^2 l!}
\\ & = \frac{(-1)^k \binom{\mu}{k} (\mu !)^2 {}_3 F_2 (-k, -m - \mu - 1, m - \mu ; - \mu, - \mu ; 1)}{(\mu - m)! (\mu + m + 1)!}
\end{flalign*}
This proves our Theorem.
\end{proof}

We note that we have used the eigenvalues of the Berezin transform $B_{\mu}$ to compute the eigenvalues of $E_{\mu,k}$. For the non-compact dual of $\mathbb{CP}^1$, namely the open unit disk and for general bounded symmetric domains, Unterberger and Upmeier \cite{UU} have found the eigenvalues of the Berezin transform. We shall study the corresponding questions about quantum channels for bounded symmetric domains in a future work.


\begin{thebibliography}{13}
\bibitem{alieng} Ali, S. T., \& Englis, M. (2005). Quantization methods: a guide for physicists and analysts. Reviews in Mathematical Physics, 17(04), 391-490.
\bibitem{andaskroy} Askey, R., Andrews, G. E., \& Roy, R. (1999). Special functions. Encyclopedia of Mathematics and its Applications, 71.
\bibitem{BMSCMP} Bordemann, M., Meinrenken, E., \& Schlichenmaier, M. (1994). Toeplitz quantization of Kähler manifolds and $\mathrm{gl}(N)$, $N \rightarrow \infty$ limits. Communications in Mathematical Physics, 165, 281-296.
\bibitem{frank} Frank, R. L. (2023). Sharp inequalities for coherent states and their optimizers. Advanced Nonlinear Studies, 23(1), 20220050.
\bibitem{hall} Hall, B. C. (2013). Lie groups, Lie algebras, and representations (pp. 333-366). Springer New York.
\bibitem{helg} Helgason, S. (1979). Differential geometry, Lie groups, and symmetric spaces. Academic press.
\bibitem{helggga} Helgason, S. (2022). Groups and geometric analysis: integral geometry, invariant differential operators, and spherical functions (Vol. 83). American Mathematical Society.
\bibitem{hump} Humphreys, J. E. (2012). Introduction to Lie algebras and representation theory (Vol. 9). Springer Science \& Business Media.
\bibitem{kuli} Kulikov, A. (2022). Functionals with extrema at reproducing kernels. Geometric and Functional Analysis, 32(4), 938-949.
\bibitem{lieb} Lieb, E. H. (1978). Proof of an entropy conjecture of Wehrl. Communications in Mathematical Physics, 62(1), 35-41.
\bibitem{liebsolBloch} Lieb, E. H., \& Solovej, J. P. (2014). Proof of an entropy conjecture for Bloch coherent spin states and its generalizations.
\bibitem{liebsolSymm} Lieb, E. H., \& Solovej, J. P. (2016). Proof of the Wehrl-type Entropy Conjecture for Symmetric $SU(N)$ Coherent States. Communications in Mathematical Physics, 348, 567-578.
\bibitem{liebsolSU11} Lieb, E. H., \& Solovej, J. P. (2019). Wehrl-type coherent state entropy inequalities for $ SU (1, 1) $ and its $ AX+ B $ subgroup. arXiv preprint arXiv:1906.00223.
\bibitem{schlich} Karabegov, A. V., \& Schlichenmaier, M. (2001). Identification of Berezin-Toeplitz deformation quantization.
\bibitem{sugi} Sugita, A. (2002). Proof of the generalized Lieb-Wehrl conjecture for integer indices larger than one. Journal of Physics A: Mathematical and General, 35(42), L621.
\bibitem{UU} Unterberger, A., \& Upmeier, H. (1994). The Berezin transform and invariant differential operators. Communications in mathematical physics, 164, 563-597.
\bibitem{wehrl} Wehrl, A. (1979). On the relation between classical and quantum-mechanical entropy. Reports on Mathematical Physics, 16(3), 353-358.
\bibitem{zag} Zagier, D. (1994). Modular forms and differential operators. Proceedings Mathematical Sciences, 104, 57-75.
\bibitem{zhangBer} Zhang, G. (1998). Berezin transform on compact Hermitian symmetric spaces. Manuscripta mathematica, 97, 371-388.
\bibitem{zhangCon} Zhang, G. (2024). Wehrl-type inequalities for Bergman spaces on
domains in C and completely positive map. Springer Proceedings in Mathematics and Statistics 447. To appear.
\end{thebibliography}
\end{document}